\documentclass{amsart}

\usepackage{amsmath,amssymb,amsthm,amsfonts,mathtools}
\usepackage{mathrsfs}
\usepackage{microtype}
\usepackage{pinlabel}
\usepackage{enumitem}
\usepackage{graphicx}
\usepackage{needspace}
\usepackage[dvipsnames]{xcolor}
\usepackage{tikz-cd}
\usepackage{url}
\usepackage{hyperref}
\usepackage[nameinlink]{cleveref}
\usepackage[hyphenbreaks]{breakurl}

\hypersetup{
  colorlinks=true,
  breaklinks,
  linkcolor=[RGB]{51 102 204},
  filecolor=Orchid,
  urlcolor=[RGB]{51 102 204},
  citecolor=Orchid,
  linktoc=all,
}

\usepackage[
backend=biber,
style=alphabetic,
maxnames=50,
]{biblatex}

\renewbibmacro{in:}{}
\AtEveryBibitem{%
  \clearfield{issn}%
  \clearfield{doi}%
  \clearfield{url}%
  \ifentrytype{book}{
    \clearfield{pages}%
  }
}

\theoremstyle{plain}
\newtheorem{theorem}{Theorem}[section]
\newtheorem{maintheorem}{Theorem}

\Crefname{maintheorem}{theorem}{theorems}
\Crefname{maintheorem}{Theorem}{Theorems}
\newtheorem{maincorollary}[maintheorem]{Corollary}

\newtheorem{proposition}[theorem]{Proposition}
\newtheorem{lemma}[theorem]{Lemma}

\theoremstyle{definition}
\newtheorem{definition}[theorem]{Definition}

\newtheorem{remark}[theorem]{Remark}

\newcommand{\ZZ}{\mathbb{Z}}

\newcommand{\BI}{\mathcal{BI}}

\DeclareMathOperator{\Aut}{Aut}
\DeclareMathOperator{\GL}{GL}
\DeclareMathOperator{\SL}{SL}
\DeclareMathOperator{\Sp}{Sp}
\DeclareMathOperator{\spm}{\mathfrak{sp}}

\title{Generators for the level $m$ congruence subgroups of braid groups}

\author{Ishan Banerjee}
\address{Ishan Banerjee: Department of Mathematics, Ohio State University, 100 Math Tower, 231 W 18th Ave, Columbus, OH 43210}
\email{banerjee.238@osu.edu}
\author{Peter Huxford}
\address{Peter Huxford: Department of Mathematics, University of Chicago, 5734 S. University Ave., Chicago, IL 60637}
\email{pjhuxford@uchicago.edu}

\begin{document}

\maketitle

\begin{abstract}
  We prove for $m\geq1$ and $n\geq5$ that the level $m$ congruence subgroup $B_n[m]$ of the braid group $B_n$ associated to the integral Burau representation $B_n\to\GL_n(\ZZ)$ is generated by $m$th powers of half-twists and the braid Torelli group. This solves a problem of Margalit, generalizing work of Assion, Brendle--Margalit, Nakamura, Stylianakis and Wajnryb.
\end{abstract}

\section{Introduction}
Let $m\geq1$, $n\geq2$, and let $B_n$ be the braid group on $n$ strands. Our main result relates three important subgroups of $B_n$ that we define below.

\subsection*{The half-twist $\boldsymbol{m}$th power subgroup}
Let $\sigma_1,\ldots,\sigma_{n-1}$ be the standard generators of $B_n$. A \emph{half-twist} is any element of $B_n$ that is conjugate to $\sigma_1$. Note that the standard generators of $B_n$ are all half-twists. We define the \emph{half-twist $m$th power subgroup} $B_n^m$ to be the normal closure of $\sigma_1^m$ in $B_n$:
\[
  B_n^m \coloneqq \langle \sigma^m : \sigma\in B_n \text{ is a half-twist} \rangle.
\]
This conflicts with the common notation $G^m$, which typically denotes the subgroup of $G$ generated by all its $m$th powers. As such, we do not make use of this more common notation in our paper.

\subsection*{The braid Torelli group}
The \emph{unreduced Burau representation} is the homomorphism $B_n\to\GL_n(\ZZ[t,t^{-1}])$ defined by $\sigma_i \mapsto I_{i-1} \oplus \begin{psmallmatrix} 1 - t & 1 \\ t & 0 \end{psmallmatrix} \oplus I_{n-i-1}$, $1\leq i<n$, where $I_j$ is the $j\times j$ identity matrix. The \emph{integral Burau representation} is the homomorphism $\rho_n\colon B_n\to\GL_n(\ZZ)$ obtained by specializing the unreduced Burau representation at $t=-1$. The \emph{braid Torelli group} $\BI_n$ is the kernel
\[
  \BI_n \coloneqq \ker\left(B_n\xrightarrow{\rho_n}\GL_n(\ZZ)\right).
\]
The Burau representation is often defined using the transposes of the matrices we use, however this difference does not affect any of the subgroups we define.

\subsection*{The level $\boldsymbol{m}$ congruence subgroup}
By post-composing the integral Burau representation with the mod-$m$ reduction map $\GL_n(\ZZ)\to\GL_n(\ZZ/m\ZZ)$, we obtain a homomorphism $B_n\to\GL_n(\ZZ/m\ZZ)$. Its kernel $B_n[m]$ is the \emph{level $m$ congruence subgroup} of $B_n$:
\[
  B_n[m] \coloneqq \ker\left(B_n\xrightarrow{\rho_n}\GL_n(\ZZ)\xrightarrow{\bmod m}\GL_n(\ZZ/m\ZZ)\right).
\]

We are now equipped to state our main theorem.

\begin{maintheorem}\label{maintheorem}
  Suppose $m\geq1$, $n\geq2$, and $m<6$ if $n=3,4$. The level $m$ congruence subgroup $B_n[m]$ of $B_n$ is generated by the braid Torelli group $\BI_n$ and $B_n^m$:
  \[
    B_n[m] = \langle \BI_n, B_n^m \rangle.
  \]
\end{maintheorem}

The equality $B_n[m]=\langle\BI_n,B_n^m\rangle$ does not hold when $n=3,4$ and $m\geq6$; in fact $\langle\BI_n,B_n^m\rangle$ is an infinite index subgroup of $B_n$ in these cases (see \Cref{infinite-index}).

We cannot omit the group $\BI_n$ from \Cref{maintheorem} in general. We discuss when $B_n[m]=B_n^m$ in \Cref{coxeter-remark}.

Brendle, Margalit, and Putman prove that $\BI_n$ is generated by the squares of Dehn twists about curves that surround 3 or 5 points \cite[Theorem C]{BMP15}. Combining their result with \Cref{maintheorem} we obtain a natural generating set of $B_n[m]$, solving a problem of Margalit \cite[Problem 3.5]{Mar19} for braid groups on at least five strands. In particular, we find that $B_n[m]$ for $n\geq5$ can be normally generated by three elements.

\begin{maincorollary}\label{maincorollary}
  Suppose $m\geq1$ and $n\geq2$. The level $m$ congruence subgroup $B_n[m]$ of $B_n$ is normally generated by the following braids:
  \begin{center}
    \begin{tabular}{@{$\bullet$ }ll}
    $\sigma_1^m$ & if $n=2$, \\
    $\sigma_1^m$, and $(\sigma_1\sigma_2)^6$ & if $n=3,4$ and $m<6$, \\
    $\sigma_1^m$, $(\sigma_1\sigma_2)^6$, and $(\sigma_1\sigma_2\sigma_3\sigma_4)^{10}$ & if $n\geq5$.
  \end{tabular}
  \end{center}
\end{maincorollary}

For $n\geq5$, \Cref{maintheorem} is an analog of Mennicke's result \cite[Section 10]{Men65} that for $m\geq1$ and $g\geq2$ the principal level $m$ congruence subgroup $\Sp_{2g}(\ZZ)[m]$ of the integral symplectic group $\Sp_{2g}(\ZZ)$ is generated by $m$th powers of transvections. Indeed, A'Campo \cite{AC79} showed that $\rho_{2g+1}(B_{2g+1})$ may be regarded as a proper, finite index subgroup of $\Sp_{2g}(\ZZ)$ for $g\geq2$, with a similar result for $\rho_{2g+2}$. This finite index difference is why \Cref{maintheorem} does not follow from Mennicke's result, although it does play a crucial role in our proof.

The situation for $n=3,4$ is more delicate; this is related to the failure of $\SL_2(\ZZ)$ to satisfy the congruence subgroup property. Let $\bar{\rho}_3\colon B_3\to\SL_2(\ZZ)$ denote the surjection defined by $\sigma_1\mapsto\begin{psmallmatrix}1&1\\0&1\end{psmallmatrix}$ and $\sigma_2\mapsto\begin{psmallmatrix}1&0\\-1&1\end{psmallmatrix}$. (The connection between $\bar{\rho}_3$ and $\rho_3$ is explained by \Cref{burau-as-symplectic}.) We give a partial answer to Margalit's problem on $B_n[m]$ for $n=3,4$ and $m\geq6$ \cite[Problem 3.5]{Mar19}.

\begin{maintheorem}\label{maintheorem-for-3-and-4}
  Let $m\geq1$, and $R\subseteq B_3[m]$ be such that the normal closure of $\bar{\rho}_3(R)$ in $\SL_2(\ZZ)$ contains $\SL_2(\ZZ)[2m]$. Then, for $n=3,4$, the level $m$ congruence subgroup $B_n[m]$ is normally generated in $B_n$ by $R$, $\sigma_1^m$, and $(\sigma_1\sigma_2)^6$.
\end{maintheorem}

One way to construct $R$ satisfying the conditions of \Cref{maintheorem-for-3-and-4} is to find a normal generating set of $\SL_2(\ZZ)[m]$. This can be extracted from a finite presentation of $\SL_2(\ZZ/m\ZZ)$, such as the one explicitly written down for all $m\geq1$ by Hsu \cite[Lemma 3.5]{Hsu96}.

It follows from work of Miller \cite{Mil02}, see also \cite[Chapter VII, Section 13]{New72}, that when $1\leq m\leq 5$ the image of $\bar{\rho}_3(\sigma_1^m)$ normally generates $\SL_2(\ZZ)[m]$ in $\SL_2(\ZZ)$. Thus \Cref{maintheorem-for-3-and-4} implies \Cref{maintheorem} in the special cases $n=3,4$ and $1\leq m\leq 5$.

\subsection*{Related work}

For the cases $n=3$ and $1\leq m\leq5$, \Cref{maintheorem} follows from Miller \cite{Mil02}. \Cref{maintheorem} was also implicitly known in the special cases where $B_n[m]=B_n^m$, however in general this equality does not hold (see \Cref{coxeter-remark}). A notable case where this equality does hold is $B_n[2]=B_n^2$, which follows from Arnol'd's identification of $B_n[2]$ with the pure braid group $PB_n$ \cite{Arn68}, and Moore's presentation of the symmetric group $S_n$ \cite{Moo97}.

For all $n\geq2$, a normal generating set of $B_n[3]$ was obtained by Assion \cite{Ass78a}, which was extended to a normal generating set of $B_n[p]$ for odd primes $p$ by Wajnryb \cite{Waj91}. Stylianakis has also found a normal generating set of $B_n[2p]$, and outlines how to construct normal generating sets for $B_n[2m]$ for square-free $m$ \cite[Corollary 5.4]{Sty18}.

Brendle and Margalit showed for all $n\geq2$ that $B_n[4]$ is generated by all squares of Dehn twists \cite{BM18}; our \Cref{maintheorem} implies for all $n\geq2$ that $B_n[4]$ is generated by just the squares of Dehn twists about all curves surrounding 2, 3, or 5 points. Nakamura has also produced a different normal generating set of $B_n[4]$ \cite[Theorem 1.3]{Nak22}.

The groups $B_n[m]$ are all finitely generated, but the problem of determining explicit finite generating sets of these groups remains open in general. This has been done for the pure braid group $B_n[2]$ by Artin \cite{Art25}, and for $B_3[p]$ where $p$ is an odd prime by Bellingeri, Damiani, Ocampo, and Stylianakis \cite[Theorem 3.9]{BDOS24}. In fact, these papers give finite presentations on the generators they provide.

We would also like to draw the reader's attention to related work of Bloomquist, Patzt, and Scherich, who completed the determination of the images of $B_n$ and $B_n[m]$ under the integral Burau representation \cite{BPS23}.

\begin{remark}\label{infinite-index}
  When $n=3,4$ and $m\geq6$ the subgroup $\langle\BI_n,B_n^m\rangle$ is not equal to $B_n[m]$; in fact it has infinite index in $B_n$. Indeed, for $n=3,4$, the braid Torelli group $\BI_n$ is normally generated in $B_n$ by $(\sigma_1\sigma_2)^6$ by \cite[Theorem C]{BMP15}. Therefore, the special homomorphism $B_4\to B_3$ defined by $\sigma_1,\sigma_3\mapsto\sigma_1$ and $\sigma_2\mapsto\sigma_2$ descends to a surjective homomorphism
  \[
    B_4/\langle\BI_4,B_4^m\rangle \to B_3/\langle \BI_3, B_3^m\rangle.
  \]
  Since $Z(B_3)=\langle(\sigma_1\sigma_2)^3\rangle$, the quotient $B_3/\langle\BI_3,B_3^m\rangle$ surjects onto $B_3/\langle Z(B_3),B_3^m\rangle$. On the generators $x=(\sigma_1\sigma_2)^{-1}$ and $y=\sigma_1\sigma_2\sigma_1$ of $B_3$, the relation $\sigma_1\sigma_2\sigma_1=\sigma_2\sigma_1\sigma_2$ is equivalent to $x^{-3}=y^2$. We also have the identity $\sigma_1=xy$. Hence
  \[
    B_3/\langle Z(B_3), B_3^m \rangle \cong \langle x,y \mid x^3 = y^2 = (xy)^m = 1 \rangle.
  \]
  This group, known as the $(2,3,m)$-von Dyck group, is infinite for $m\geq6$ by a result of Miller \cite{Mil02}, see also \cite[Section 5.3]{CM57}. Therefore, when $n=3,4$ and $m\geq6$ the subgroup $\langle\BI_n,B_n^m\rangle$ has infinite index in $B_n$.
\end{remark}

\begin{remark}\label{coxeter-remark}
  In \Cref{mth-power-contained-in-mth-congruence} we observe that $B_n^m\leq B_n[m]$ for all $m\geq1$ and $n\geq2$. Thus, $B_n^m=B_n[m]$ if and only if $B_n/B_n^m$ and $B_n/B_n[m]$ have the same order. The orders of the quotients $B_n/B_n^m$ and $B_n/B_n[m]$ are all known. The determination of the orders of $B_n/B_n^m$ was completed by Coxeter \cite{Cox59}, and the same task for $B_n/B_n[m]$ was completed by Bloomquist, Patzt and Scherich \cite{BPS23}.

  Coxeter showed that $B_n/B_n^m$ is finite if and only if $1/m+1/n>1/2$ \cite{Cox59}, see also \cite{Ass78b}. By comparing the orders of the quotients $B_n/B_n^m$ and $B_n/B_n[m]$ for the remaining values of $m\geq1$ and $n\geq2$ for which $B_n/B_n^m$ is finite, we find that
  
  \begin{itemize}
  \item $B_n[m]=B_n^m$ if and only if $m=1,2$, or $n=2$, or $m=3$ and $n=3,4$,
  \item $B_3[4]/B_3^4\cong\ZZ/2\ZZ$, $B_3[5]/B_3^5\cong\ZZ/5\ZZ$, and $ B_5[3]/B_5^3\cong\ZZ/3\ZZ$,
  \item $B_n[m]/B_n^m$ is infinite otherwise.
  \end{itemize}
\end{remark}

\subsection*{Outline of the paper}

In \Cref{burau-background,sp2g-background} we give background on the Burau representation and the integral symplectic group $\Sp_{2g}(\ZZ)$. In \Cref{stepping-stone} we prove an intermediate result: $B_n[m]=\langle B_n[2m],B_n^m\rangle$. In \Cref{reduction} we use this to reduce \Cref{maintheorem,maintheorem-for-3-and-4} to \Cref{final-theorem}, which asserts that $\rho_n(B_n^m)$ contains all $2m$th powers of certain elements of $\GL_n(\ZZ)$. Finally, we prove \Cref{final-theorem} in \Cref{final-proof}.

\subsection*{Acknowledgements} We would like to thank Benson Farb, Dan Margalit, and Charalampos Stylianakis for their helpful comments on the paper.

\section{The integral Burau representation as a symplectic representation}\label{burau-background}

There is a sense in which the integral Burau representation $\rho_n\colon B_n\to\GL_n(\ZZ)$ is a symplectic representation of the braid group, which we explain in \Cref{burau-as-symplectic}. To state this \namecref{burau-as-symplectic}, we first establish some notation that we use throughout the paper. We write
\[
  \Gamma_n \coloneqq \rho_n(B_n), \quad \Gamma_n^m \coloneqq \rho_n(B_n^m), \quad \Gamma_n[m] \coloneqq \rho_n(B_n[m]) = \Gamma_n\cap\GL_n(\ZZ)[m].
\]

Note that we have the short exact sequence
\[
  1 \to \BI_n \to B_n[m] \to \Gamma_n[m] \to 1,
\]
and thus $B_n[m]=\langle\BI_n,B_n^m\rangle$ is equivalent to $\Gamma_n[m]=\Gamma_n^m$.

\begin{definition}[\textbf{Transvections}]
  Given a $\ZZ$-module $V$ equipped with an alternating $\ZZ$-bilinear form $\langle-,-\rangle$, for each $x\in V$ we define the \emph{transvection} $T_x\in\Aut(V,\langle-,-\rangle)$ by the formula
\[
  T_x(v) \coloneqq v - \langle v,x\rangle x \quad \forall v\in V.
\]
\end{definition}

Note that
\[
  AT_xA^{-1}=T_{Ax} \quad \text{for all } x\in V \text{ and } A\in\Aut(V,\langle-,-\rangle).
\]

Now let $e_1,\ldots,e_n$ denote the standard basis of $\ZZ^n$, and equip $\ZZ^n$ with the alternating $\ZZ$-bilinear form defined by
\[
  \langle e_i,e_j\rangle = \begin{cases} 1 & i<j \\ 0 & i=j \\ -1 & i>j. \end{cases}
\]
Define $c_i\coloneqq e_i-e_{i+1}$ for $1\leq i<n$. Then we have the following equivalent definition of the integral Burau representation that is easily checked.
\begin{proposition}
  For $n\geq2$, the integral Burau representation $\rho_n$ satisfies
  \[
    \rho_n(\sigma_i) = T_{c_i} \quad 1\leq i<n.
  \]
\end{proposition}

This provides a quick way to verify the following \namecref{mth-power-contained-in-mth-congruence}.

\begin{proposition}\label{mth-power-contained-in-mth-congruence}
  If $m\geq1$ and $n\geq2$ then $B_n^m\leq B_n[m]$.
\end{proposition}

\begin{proof}
  If $T$ is a transvection then $T^m$ is the identity mod $m$, hence $B_n^m\leq B_n[m]$.
\end{proof}

We also define the following sublattices of $\ZZ^n$.

\begin{definition}[\textbf{The sublattices $W_n$ and $V_n$}]
  Let $n\geq2$, and let $W_n$ be the sublattice of $\ZZ^n$ given by
  \[
    W_n \coloneqq \ZZ c_1 + \cdots + \ZZ c_{n-1} = \{\lambda_1e_1 + \cdots + \lambda_ne_n \in \ZZ^n : \lambda_1 + \cdots + \lambda_n = 0 \}.
  \]
  Let $V_n$ be the sublattice of $\ZZ^n$ given by
  \[
    V_n \coloneqq \begin{cases} \ZZ^n & n \text{ even} \\ W_n & n \text{ odd.} \end{cases}
  \]
  Note that $\mathrm{rank}(V_n)=2\lfloor n/2\rfloor$ is even for all $n\geq2$.
\end{definition}

The following \namecref{burau-as-symplectic} is well-known, but is typically expressed with different notation.

\needspace{3\baselineskip} 
\begin{proposition}\label{burau-as-symplectic}
  Let $n\geq2$.
  \begin{enumerate}[label=(\alph*)]
  \item The sublattices $W_n$ and $V_n$ are $\Gamma_n$-invariant.
  \item The bilinear form $\langle-,-\rangle$ on $\ZZ^n$ restricts to a symplectic form on $V_n$.
  \item The vector
    \[
      w \coloneqq e_1 - e_2 + \cdots + (-1)^{n-1}e_n \in \ZZ^n
    \]
    is fixed by $\Gamma_n$.
  \item If $n$ is odd, then $\ZZ^n$ decomposes as a sum $V_n+\ZZ w$.
  \item If $n$ is even, then
    \[
      w = e_1 - e_2 + \cdots - e_n = c_1 + c_3 + \cdots + c_{n-1} \in W_n,
    \]
    and the orthogonal complement $w^\perp$ of $w$ in $\ZZ^n$ is $W_n$.
  \item The natural map
    \[
      \Gamma_n \to \Sp(V_n) \cong \Sp_{2\lfloor n/2\rfloor}(\ZZ)
    \]
    is injective, and when $\Gamma_n$ is viewed as a subgroup of $\Sp(V_n)$ we have
    \[
      \Gamma_n[m] = \Gamma_n \cap \Sp(V_n)[m].
    \]
  \end{enumerate}
\end{proposition}

\begin{proof}
  Parts (a), (c), (d), and (e) are easy to verify. To prove part (b) it suffices to find a symplectic basis of $V_n$, see e.g., \cite[Section 2.1]{BM18} or the proof of \cite[Lemma 2.1]{BPS23}. Part (f) is immediate when $n$ is even, and when $n$ is odd it follows from parts (c) and (d).
\end{proof}

With respect to the $\ZZ$-basis $(c_1,c_2)$ of $V_3$, the images $\rho_3(\sigma_1)$ and $\rho_3(\sigma_2)$ act by the matrices $\begin{psmallmatrix}1&1\\0&1\end{psmallmatrix}$ and $\begin{psmallmatrix}1&0\\-1&1\end{psmallmatrix}$, which generate $\SL_2(\ZZ)\cong\Sp(V_3)$. Thus, by \Cref{burau-as-symplectic} we may identify $\rho_3$ with the surjection $\bar{\rho}_3\colon B_3\to\SL_2(\ZZ)$ mentioned in the statement of \Cref{maintheorem-for-3-and-4}.

\section{Background on the integral symplectic group}\label{sp2g-background}

The purpose of this section is to establish some useful facts about the integral symplectic group and its stabilizer subgroups. If a group $G$ acts on a set $X$ and $x\in X$, denote the stabilizer of $x$ in $G$ by $G_x$.

Let $g\geq2$ and equip $\ZZ^{2g}$ with a symplectic form $\langle-,-\rangle$. If $w\in\ZZ^{2g}$ is a primitive vector, then $\langle-,-\rangle$ descends to a symplectic form on $w^\perp/w\cong\ZZ^{2g-2}$.

If $T\in\Sp_{2g}(\ZZ)_w$, then $T$ preserves the sublattice $w^\perp$ of $\ZZ^{2g}$, and thus descends to an element of $\Sp(w^\perp/w)\cong\Sp_{2g-2}(\ZZ)$. This gives rise to a natural map $\Sp_{2g}(\ZZ)_w\to\Sp(w^\perp/w)$, which restricts to a homomorphism $\Sp_{2g}(\ZZ)[m]_w\to\Sp(w^\perp/w)[m]$ between the level $m$ congruence subgroups for all $m\geq1$.

\begin{proposition}\label{kernel-gen-by-transvections}
  If $g\geq2$, $m\geq1$ and $w\in\ZZ^{2g}$ is primitive then the group
  \[
    G \coloneqq \langle T_u^m : u\in\ZZ^{2g}, u\perp w\rangle
  \]
  contains the kernel of the natural map
  \[
    K \coloneqq\ker(\Sp_{2g}(\ZZ)[m]_w \to \Sp(w^\perp/w)[m]).
  \]
\end{proposition}
\begin{proof}
  Fix $v\in\ZZ^{2g}$ with $\langle v,w\rangle=1$. We first prove that $T\mapsto Tv-v$ defines a set-theoretic bijection $K\to mw^\perp$. Let $T\in K$. Since $Tw=w$, we have $\langle v,w\rangle=\langle Tv,Tw\rangle=\langle Tv,w\rangle$, and hence $\langle Tv-v,w\rangle=0$. Also, $Tv-v$ is a multiple of $m$ since $T$ is the identity modulo $m$. Hence $Tv-v\in mw^\perp$.

  To prove the map $K\to mw^\perp$ is surjective, note that for each $x\in w^\perp$ we may define $S_x\in K$ with $S_xv-v=mx$ as follows:
  \[
    S_xv \coloneqq v + mx, \qquad S_xu \coloneqq u + \langle u,x\rangle mw \quad \forall u\in w^\perp.
  \]
  For injectivity, suppose that $T\in K$ with $Tv-v=mx$, where $x\in w^\perp$. Let $T'\coloneqq S_x^{-1}T\in K$. Suppose that $u\in w^\perp$. We have $T'v=v$, hence $T'u-u\in v^\perp$. Furthermore, $T'u-u\in\ZZ w$, since $T'$ acts trivially on $w^\perp/w$. Therefore $T'u-u=0$, and so $T'$ is the identity. Thus $T=S_x$, which proves injectivity.

  Next, note that $T_w^m\in G\cap K$, and also if $T\in K$ and $k\in\ZZ$, then
  \[
    T_w^{k}Tv - v = T_w^k(Tv-v) + (T_w^kv-v) = (Tv-v) - kw.
  \]
  Hence, it suffices to prove that for all $x\in w^\perp$, there exists $T\in G\cap K$ such that $Tv-v\equiv mx\pmod{mw}$. Let $x\in w^\perp$ and define $T\coloneqq T_x^mT_{x+w}^{-m}\in G\cap K$. We have
    \begin{align*}
    Tv - v &= T_x^m(v + m\langle v,x+w\rangle(x+w)) - v\\
    &= T_x^m(v + m(\langle v,x\rangle + 1)(x+w)) - v \\
    &= m(\langle v,x\rangle + 1)(x+w) - m\langle v,x\rangle x \\
    &\equiv mx \pmod{mw}.\qedhere
    \end{align*}
\end{proof}

\begin{remark}
  A small addition to the above proof also shows that the kernel $K$ in \Cref{kernel-gen-by-transvections} is a 2-step nilpotent group, with center $Z(K)=\langle T_w^m\rangle$ and the mapping $T\mapsto\frac{1}{m}(Tv-v)$ inducing an isomorphism $K/Z(K)\cong w^\perp/w$.
\end{remark}

\begin{proposition}\label{gen-by-transvections}\
  \begin{enumerate}[label=(\alph*)]
  \item If $g\geq2$ and $m\geq1$ then
    \[
      \Sp_{2g}(\ZZ)[w] = \langle T_u^m : u\in\ZZ^{2g} \rangle.
    \]
  \item If $g\geq3$, $m\geq1$ and $w\in\ZZ^{2g}$ is primitive then
    \[
      \Sp_{2g}(\ZZ)[m]_w = \langle T_u^m : u\in\ZZ^{2g},\ u\perp w\rangle.
    \]
  \end{enumerate}
\end{proposition}
\begin{proof}
  Part (a) is due to Mennicke \cite[Section 10]{Men65}. For part (b), note that part (a) implies that $\Sp(w^\perp/w)[m]$ is generated by $m$th powers of transvections. Hence \Cref{kernel-gen-by-transvections} implies the result.
  \end{proof}

The following \namecref{sp2g-equivariant} is due to Newman and Smart \cite{NS64}.

\begin{proposition}[{\cite[Theorem 7]{NS64}}]\label{sp2g-equivariant}
  Let $g\geq1$ and $m,\ell\geq2$ with $m\mid\ell$. There are isomorphisms of abelian groups
  \[
    \Sp_{2g}(\ZZ)[\ell]/\Sp_{2g}(\ZZ)[m\ell]\cong\spm_{2g}(\ZZ/m\ZZ) \cong (\ZZ/m\ZZ)^{2g^2+g}.
  \]
\end{proposition}

While we do not include Newman and Smart's complete proof of the \namecref{sp2g-equivariant}, we do outline the construction of the first isomorphism. Each $A\in\Sp_{2g}(\ZZ)[\ell]$ has the form $A=I+\ell B$. It can be shown that $B\bmod m\in\spm_{2g}(\ZZ/m\ZZ)$. Furthermore, the mapping $A=I+\ell B\mapsto B\bmod m$ defines a surjective homomorphism $\Sp_{2g}(\ZZ)[\ell]\to\spm_{2g}(\ZZ/m\ZZ)$ whose kernel is $\Sp_{2g}(\ZZ)[m\ell]$.

In particular, if $T\in\Sp_{2g}(\ZZ)$ is a transvection, then $T^\ell\in\Sp_{2g}(\ZZ)[\ell]$ is carried to $T-I\bmod m\in\spm_{2g}(\ZZ/m\ZZ)$ by this mapping.

\section{Level $m$ from level $2m$ and $m$th powers}\label{stepping-stone}

As a stepping stone towards our ultimate goal of proving that $\Gamma_n[m]=\Gamma_n^m$ for the values of $m$ and $n$ specified in \Cref{maintheorem}, in this section we prove the following equality.

\begin{theorem}\label{level-m-from-level-2m-and-mth-powers}
  If $m\geq1$ and $n\geq2$ then
  \[
    \langle\Gamma_n^m, \Gamma_n[2m]\rangle = \Gamma_n[m].
  \]
\end{theorem}

With this theorem in hand, the equality $\Gamma_n[m]=\Gamma_n^m$ is clearly equivalent to the inclusion $\Gamma_n[2m]\leq\Gamma_n^m$, which we consider in later sections.

\Cref{level-m-from-level-2m-and-mth-powers} also immediately implies \Cref{maintheorem-for-3-and-4} in the case of $n=3$, since the isomorphism $\Gamma_3\cong\SL_2(\ZZ)$ given by \Cref{burau-as-symplectic} takes $\Gamma_3[2m]$ to $\SL_2(\ZZ)[2m]$. To handle the $n=4$ case we will need results from subsequent sections.

\subsection{Proof when $m$ is a power of 2}

We being by proving the following special case of \Cref{level-m-from-level-2m-and-mth-powers}.

\begin{lemma}\label{m-power-of-2}
  If $r\geq0$ and $n\geq2$, then
  \[
    \langle\Gamma_n^{2^r}, \Gamma_n[2^{r+1}] \rangle = \Gamma_n[2^r].
  \]
\end{lemma}

\begin{proof}
  When $r=0$ this follows from the equalities $\Gamma_n^1=\Gamma_n=\Gamma_n[1]$. Assume now that $r\geq1$. Let $\Gamma\coloneqq\langle\Gamma_n^{2^r},\Gamma_n[2^{r+1}]\rangle\leq\Gamma_n$. It suffices to show that $\Gamma/\Gamma_n[2^{r+1}]=\Gamma_n[2^r]/\Gamma_n[2^{r+1}]$.

  By \Cref{burau-as-symplectic} we may identify $\Gamma_n$ with its image in $\Sp(V_n)$, and hence $\Gamma_n[2^r]/\Gamma_n[2^{r+1}]$ with a subgroup of the quotient group $\Sp(V_n)[2^r]/\Sp(V_n)[2^{r+1}]$. The latter group is abelian and isomorphic to $\spm(V_n/2V_n)$ by \Cref{sp2g-equivariant}. Furthermore, the construction of the isomorphism is such that if $T\in\Sp(V_n)$ is a transvection, then the image of $T^{2^r}$ in $\Sp(V_n)[2^r]/\Sp(V_n)[2^{r+1}]$ is carried to the element $T-I\bmod 2$ of $\spm(V_n/2V_n)$.

  It follows that for each $n\geq2$, the correctness of \Cref{m-power-of-2} does not depend on the value of $r\geq1$. Hence, it suffices to prove it when $r=1$. This special case follows from the equality $B_n^2=B_n[2]$, which holds by combining work of Moore \cite{Moo97} and Arnol'd \cite{Arn68} as mentioned in the introduction.
\end{proof}

\subsection{Proof of the general case}

\begin{proof}[Proof of \Cref{level-m-from-level-2m-and-mth-powers}]
  Let $m\geq1$ and $n\geq2$. Write $m=2^rs$ where $s$ is odd. Let $\Gamma\coloneqq\langle\Gamma_n^m,\Gamma_n[2m]\rangle$. It suffices to show that $\Gamma/\Gamma_n[2m]=\Gamma_n[m]/\Gamma_n[2m]$.

  By the Chinese remainder theorem, the natural reduction maps define an isomorphism
  \[
    \Gamma_n/\Gamma_n[2m] \cong \Gamma_n/\Gamma_n[2^{r+1}] \times \Gamma_n/\Gamma_n[s].
  \]
  Recall that $\Gamma_n^m\leq\Gamma_n[m]$ by \Cref{mth-power-contained-in-mth-congruence}. The image of $\Gamma_n[m]$ in $\Gamma_n/\Gamma_n[s]$ is trivial since $s\mid m$. Hence, it suffices to show that $\Gamma_n^m$ and $\Gamma_n[m]$ have the same image in $\Gamma_n/\Gamma_n[2^{r+1}]$.

  If $T\in\Gamma_n$ is a transvection, then $T^m$ and $T^{2^r}$ agree modulo $2^{r+1}$. Hence $\Gamma_n^m$ and $\Gamma_n^{2^r}$ have the same image in $\Gamma_n/\Gamma_n[2^{r+1}]$. By \Cref{m-power-of-2} this implies $\Gamma_n^m$ and $\Gamma_n[2^r]$ have the same image in $\Gamma_n/\Gamma_n[2^{r+1}]$. Since $\Gamma_n^m\leq\Gamma_n[m]\leq\Gamma_n[2^r]$, we find that $\Gamma_n^m$ and $\Gamma_n[m]$ also have the same image in $\Gamma_n/\Gamma_n[2^{r+1}]$, as required.
\end{proof}

\section{Reduction of the main theorems to a result on transvections}\label{reduction}

In this section we reduce \Cref{maintheorem,maintheorem-for-3-and-4} to the following \namecref{final-theorem}.

\begin{theorem}\label{final-theorem}
  If $m\geq1$ and $n\geq2$, then $T_x^{2m}\in\Gamma_n^m$ for all $x\in W_n$.
\end{theorem}

We prove \Cref{final-theorem} in \Cref{final-proof}.

\subsection{Proof of \Cref{maintheorem-for-3-and-4} from \Cref{final-theorem}}

\begin{proof}
  As mentioned in \Cref{stepping-stone}, the case $n=3$ immediately follows from \Cref{level-m-from-level-2m-and-mth-powers}, so it remains to handle $n=4$.

  Let $m\geq1$ and let $R\subseteq B_3[m]$ be such that the normal closure of $\bar{\rho}_3(R)$ in $\SL_2(\ZZ)$ contains $\SL_2(\ZZ)[2m]$. Let $\Gamma$ be the normal closure of $\rho_4(R)$ in $\Gamma_4$.

  Since $\BI_4$ is normally generated by $(\sigma_1\sigma_2)^6$ by \cite[Theorem C]{BMP15}, it suffices to show that $\langle\Gamma,\Gamma_4^m\rangle=\Gamma_4[m]$. Hence, by \Cref{level-m-from-level-2m-and-mth-powers} it suffices to show that $\Gamma_4[2m]\leq\langle\Gamma,\Gamma_4^m\rangle$.

  Recall that we may view $\Gamma_4$ as a subgroup of $\Sp(V_4)_w$ by \Cref{burau-as-symplectic}. As discussed in \Cref{sp2g-background}, there is a natural homomorphism $\Sp(V_4)[2m]_w\to\Sp(w^\perp/w)[2m]$. By \Cref{kernel-gen-by-transvections} and \Cref{final-theorem}, $\Gamma_4^m$ contains the kernel of this homomorphism. So it suffices to show that the image of $\Gamma$ in $\Sp(w^\perp/w)$ contains $\Sp(w^\perp/w)[2m]$.

  To prove this, we note that the following diagram commutes.
  \[\begin{tikzcd}
      {B_4} & {\Gamma_4} &[-2em] {\Sp(V_4)_w} \\
      {B_3} & {\SL_2(\ZZ)} & {\Sp(w^\perp/w)}
      \arrow["{\rho_4}", from=1-1, to=1-2]
      \arrow["\varphi"', from=1-1, to=2-1]
      \arrow[hook, from=1-2, to=1-3]
      \arrow[from=1-3, to=2-3]
      \arrow["{\bar{\rho}_3}"', from=2-1, to=2-2]
      \arrow["\cong"{description}, draw=none, from=2-3, to=2-2]
    \end{tikzcd}\]

  Here $\varphi\colon B_4\to B_3$ denotes the special homomorphism given by $\sigma_1,\sigma_3\mapsto\sigma_1$ and $\sigma_2\mapsto\sigma_2$, and the isomorphism $\Sp(w^\perp/w)\cong\SL_2(\ZZ)$ arises from using the images of $c_1$ and $c_2$ in $w^\perp/w$ as a $\ZZ$-basis to identify $w^\perp/w$ with $\ZZ^2$.

  Since the standard inclusion $B_3\to B_4$ is a section of $\varphi$, the hypothesis that the normal closure of $\rho_3(R)$ in $\SL_2(\ZZ)$ contains $\SL_2(\ZZ)[2m]$ implies that the image of $\Gamma$ in $\Sp(w^\perp/w)$ contains $\Sp(w^\perp/w)[2m]$, as required.
\end{proof}

\subsection{Proof of \Cref{maintheorem} from \Cref{final-theorem}}
  
\begin{proof}
  When $n=2$ we have $B_2[m]=B_2^m$ for all $m\geq1$ (see \Cref{coxeter-remark}). As noted in the introduction, when combined with Miller's result \cite{Mil02}, \Cref{maintheorem-for-3-and-4} implies \Cref{maintheorem} in the cases $n=3,4$.

  Assume now that $n\geq5$. Recall that by \Cref{burau-as-symplectic} we may view $\Gamma_n$ as a subgroup of $\Sp(V_n)$, and if $n$ is even then $\Gamma_n$ lies in the stabilizer subgroup $\Sp(V_n)_w$. By \Cref{burau-as-symplectic,gen-by-transvections} we have
  \[
    \langle T_x^{2m} : x\in W_n \rangle = \begin{cases} \Sp(V_n)[2m] & n \text{ odd} \\ \Sp(V_n)[2m]_w & n \text{ even.} \end{cases}
  \]
  Hence, \Cref{final-theorem} implies that $\Gamma_n[2m]=\langle T_x^{2m}:x\in W_n\rangle\leq\Gamma_n^m$. Therefore, by \Cref{level-m-from-level-2m-and-mth-powers} we have $\Gamma_n[m]=\Gamma_n^m$, and hence $B_n[m]=\langle\BI_n,B_n^m\rangle$.
\end{proof}

\section{Proof of the result on transvections}\label{final-proof}

It remains to prove \Cref{final-theorem}, namely that for all $m\geq1$ and $n\geq2$
\[
  T_x^{2m} \in \Gamma_n^m \qquad \forall x\in W_n=\ZZ c_1+\cdots+\ZZ c_{n-1}.
\]

Many of the ingredients in the proof of \Cref{final-theorem} are adapted from arguments that appear in Section 3 of Janssen's thesis \cite{Jan85a}. Janssen only considers the case of $m=1$, but works in a more general setting than the integral Burau representation. Our main contributions are noticing that some proofs there can be adapted to work for $m>1$, as well as simplifying the arguments for the special case of the integral Burau representation.

\begin{proof}[Proof of \Cref{final-theorem}]
  We proceed by induction on $n$. For $n=2$ we have $W_2=\ZZ c_1$, hence the following stronger statement holds: $T_x^m\in\Gamma_2^m$ for all $x\in W_2$. Assume now that $n\geq3$.

  Let $x\in W_n$ and write $x=\lambda_1c_1+\cdots+\lambda_{n-1}c_{n-1}$, where $\lambda_1,\ldots,\lambda_{n-1}\in\ZZ$.

  \subsection{Step one: it suffices to assume $\lambda_{n-2}=0$.}
  To reduce to the case where $\lambda_{n-2}=0$, we will construct $x'$ in the $\Gamma_n$-orbit of $x$ such that
  \[
    x' \in \ZZ c_1 + \cdots + \ZZ c_{n-3} + \ZZ c_{n-1}.
  \]
  Since $T_x$ and $T_{x'}$ will be $\Gamma_n$-conjugate, and $\Gamma_n^m$ is a normal subgroup of $\Gamma_n$, this will show that $T_x^{2m}\in\Gamma_n^m\iff T_{x'}^{2m}\in\Gamma_n^m$.

  Consider the $\ZZ$-linear map $A\colon W_n\to\ZZ^2$ defined as follows. Let
  \[
    Ac_{n-2} \coloneqq \begin{pmatrix} 0 \\ 1 \end{pmatrix}, \quad Ac_{n-1} \coloneqq\begin{pmatrix} -1 \\ 0 \end{pmatrix}, \qquad \text{and if $n\geq4$:} \quad Ac_{n-3} \coloneqq \begin{pmatrix} 1 \\ 0 \end{pmatrix},
  \]
  and $Ac_1 \coloneqq \cdots \coloneqq Ac_{n-4} \coloneqq 0$. A straightforward calculation shows that
  \[
    A \circ T_{c_{n-2}}|_{W_n} = \begin{pmatrix} 1 & 0 \\ -1 & 1 \end{pmatrix} \circ A, \qquad A \circ T_{c_{n-1}}|_{W_n} = \begin{pmatrix} 1 & 1 \\ 0 & 1 \end{pmatrix} \circ A.
  \]
  Let $d$ be the $\gcd$ of the two entries of $Ax$. By the Euclidean algorithm, there exists an element $B\in\langle T_{c_{n-2}},T_{c_{n-1}}\rangle\leq\Gamma_n$ such that $ABx=\begin{psmallmatrix}d\\0\end{psmallmatrix}$. Hence $x'\coloneqq Bx$ lies in
  \[
    A^{-1}(\ZZ\begin{psmallmatrix}1\\0\end{psmallmatrix}) = \ZZ c_1 + \cdots + \ZZ c_{n-3} + \ZZ c_{n-1}.
  \]

  \subsection{Step two: $L_+$ and $L_-$}

  By the previous section, we may assume $\lambda_{n-2}=0$. We will define self maps $L_+$ and $L_-$ on $W_n$ such that $T_x^{2m}\in\Gamma_n^m\iff T_{L_{\pm}x}^{2m}\in\Gamma_n^m$, assuming $\lambda_{n-2}=0$.

  We will make use of the following \namecref{wajnryb-prop} due to Wajnryb \cite{Waj80}.
  \begin{proposition}[{\cite[Proposition 1]{Waj80}}]\label{wajnryb-prop}
    Let $n\geq3$. If $y,y'\in W_n$ are such that
    \[
      \langle y,W_n \rangle = \langle y',W_n \rangle = \ZZ,
    \]
    and $y-y'\in 2W_n$, then $y$ and $y'$ are in the same $\Gamma_n$-orbit.
  \end{proposition}

  Now, let
  \[
    y_{\pm} = (\pm1 + 2(\lambda_{n-3}-\lambda_{n-1}))c_{n-1},
  \]
  and define $L_{\pm}x$ to be the quantity $x+y_{\pm}$. Note that $T_{y_{\pm}}^m\in\Gamma_n^m$, since $y_{\pm}\in\ZZ c_{n-1}$. Furthermore, $\langle 2x+y_{\pm},c_{n-2}\rangle=\mp1$ and $2x+y_{\pm}-c_{n-1}\in 2W_n$. By \Cref{wajnryb-prop}, $2x+y_{\pm}$ is in the $\Gamma_n$-orbit of $c_{n-1}$, and hence $T_{2x+y_{\pm}}^m\in\Gamma_n^m$.

  Since $\lambda_{n-2}=0$, we have $\langle x,y_{\pm}\rangle=0$, and hence the following identity of commuting transvections holds:
  \[
    T_x^2T_{x+y_{\pm}}^2 = T_{2x+y_{\pm}}T_{y_{\pm}} \implies T_x^{2m}T_{L_{\pm}x}^{2m} = T_{2x+y_{\pm}}^mT_{y_{\pm}}^m \in \Gamma_n^m.
  \]
  In particular, we have $T_x^{2m}\in\Gamma_n^m\iff T_{L_{\pm}x}^{2m}\in\Gamma_n^m$.

  \subsection{Step three: it suffices to also assume $\lambda_{n-1}=0$}

  Let $L_+$ and $L_-$ be the operations defined in the previous subsection. We note two easily checkable facts:
  \begin{enumerate}[label=(\roman*)]
  \item $L_+x-x$ is an odd multiple of $c_{n-1}$,
  \item $L_{\pm}L_{\mp}x=x\pm 2c_{n-1}$.
  \end{enumerate}
  It follows from (i) that it suffices to assume that $\lambda_{n-1}$ is even, and then it follows from (ii) that it suffices to assume that $\lambda_{n-1}=0$.

  \subsection{Step four: use the inductive hypothesis}

  By the previous steps, it suffices to prove for all $n\geq3$ that
  \[
    T_x^{2m} \in \Gamma_n^m \qquad \forall x\in\ZZ c_1 + \cdots + \ZZ c_{n-3}.
  \]
  When $n=3$ the sum $\ZZ c_1+\cdots+\ZZ c_{n-3}$ is the empty sum, so this is immediate. Assume now that $n\geq4$. If $\sigma\in B_{n-2}$, then $\rho_n(\sigma)=\rho_{n-2}(\sigma)\oplus\begin{psmallmatrix}1&0\\0&1\end{psmallmatrix}$. Furthermore, if $x\in\ZZ c_1+\cdots+\ZZ c_{n-3}$ then the transvection $T_x$ preserves $\ZZ e_1+\cdots+\ZZ e_{n-2}$, and fixes both $e_{n-1}$ and $e_n$.

  By induction, for all $x\in\ZZ c_1+\cdots+\ZZ c_{n-3}$ we have
  \[
    T_x^{2m} \in \left\{ \begin{pmatrix} A & 0 & 0 \\ 0 & 1 & 0 \\ 0 & 0 & 1 \end{pmatrix} : A \in \Gamma_{n-2}^m \right\} \leq \Gamma_n^m \cap \rho_n(B_{n-2}).\qedhere
  \]
\end{proof}

\printbibliography{}

\end{document}